\documentclass[11pt]{article}%
\usepackage{amsfonts}
\usepackage[a4paper,margin=1.1in]{geometry}
\usepackage[dvipsnames]{xcolor}
\usepackage{amsmath,amssymb,amsthm}
\usepackage{mathtools}
\usepackage{microtype}
\usepackage{enumitem}
\usepackage{hyperref}%
\usepackage{amsmath}%
\usepackage{amssymb}%
\usepackage{graphicx}
\providecommand{\U}[1]{\protect\rule{.1in}{.1in}}
\pdfoutput=1
\hypersetup{
colorlinks=true,
linkcolor=blue!50!black,
citecolor=blue!50!black,
urlcolor=blue!50!black
}
\theoremstyle{plain}
\newtheorem{theorem}{Theorem}[section]
\newtheorem{proposition}[theorem]{Proposition}
\newtheorem{corollary}[theorem]{Corollary}

\theoremstyle{definition}

\begin{document}

\title{Trivial Isochronous Centers in Odd Degrees: a Two--Branch Picture}
\author{J.A. Vera}
\date{\today}
\maketitle

\begin{abstract}
We revisit the characterization of \emph{trivial} isochronous centers for
planar polynomial Hamiltonian systems in degrees $5$ and $7$ obtained by
Braun--Llibre--Mereu, and we formalize two conclusions suggested by their
method. First, a \emph{triangular family} yields trivial (indeed global)
isochronous centers in every odd degree $n=2k-1\geq3$. Second, a genuinely
different \emph{quadratic--shear} ($Q$) family appears exactly when
$n\equiv3\pmod 4$, beginning at $n=7$, explaining the observed
\textquotedblleft alternating\textquotedblright\ emergence of a second branch.
For $n=9$ this second branch cannot occur by degree parity. Our statements
rest on the structure of the degree--7 proof and the general triangular
construction in the preprint, together with the standard isochrony
characterization $\mathcal{H}=\tfrac{1}{2}(f_{1}^{2}+f_{2}^{2})$ with $\det
Df\equiv1$.

\end{abstract}

\section{Setting and background}

Consider a planar polynomial Hamiltonian system
\[
\dot{x}=-\mathcal{H}_{y}(x,y),\qquad\dot{y}=\mathcal{H}_{x}(x,y),
\]
with $\mathcal{H}\in\mathbb{R}[x,y]$ and $\mathcal{H}(0,0)=0$. Following the
classical criterion, the origin is an isochronous center of period $2\pi$ if
and only if, in a neighborhood of $(0,0)$,
\begin{equation}
\mathcal{H}(x,y)=\frac{f_{1}(x,y)^{2}+f_{2}(x,y)^{2}}{2},\qquad\det
Df(x,y)\equiv1,\qquad f(0,0)=(0,0),\label{eq:iso-criterion}%
\end{equation}
with $f=(f_{1},f_{2})$ analytic. When $f$ can be chosen polynomial, we say the
center is \emph{trivial}. In the trivial case, globality reduces to the global
injectivity of $f$, linking the question to the real two--dimensional Jacobian
Conjecture (see the references in \cite{BraLliMer}).

Two tools used repeatedly in \cite{BraLliMer} are:

\begin{itemize}
\item \textbf{Homogeneous degeneracy (Lemma 1 of \cite{BraLliMer}).} If $p,q$
are homogeneous of degrees $m,n$ with $\det D(p,q)\equiv0$, then
$p=c_{p}r^{m^{\prime}}$ and $q=c_{q} r^{n^{\prime}}$ for a common homogeneous
$r $.

\item \textbf{Transport equation (Lemma 2 of \cite{BraLliMer}).} A PDE of the
form $p_{x}+\beta\,p_{y}=h$ with $p,h$ homogeneous reduces, under $y\mapsto
y-\beta x$, to a single integral in $x$.
\end{itemize}

In degrees $5$ and $7$, \cite{BraLliMer} \emph{completely} classifies trivial
centers: there is a single family in degree $5$ and two non--linearly
equivalent families in degree $7$ (their Theorems~4 and~5). Even degrees do
not admit trivial centers under the same polynomial hypothesis on $f$ (see the
discussion and references in \cite{BraLliMer}).

\section{Triangular branch in every odd degree}

\begin{theorem}
[Triangular branch]\label{thm:triangular} Let $k\geq2$ and set $n=2k-1$.
Define
\begin{equation}
f_{1}(x,y)=x+\sum_{i=2}^{k}c_{i}\,y^{i},\qquad f_{2}(x,y)=y+\lambda
\,f_{1}(x,y).\label{eq:triangular-f}%
\end{equation}
Then $\det Df\equiv1$ and the Hamiltonian
\[
\mathcal{H}(x,y)=\frac{f_{1}(x,y)^{2}+f_{2}(x,y)^{2}}{2}%
\]
generates a trivial isochronous center at the origin of degree $n=2k-1$.
Moreover, $f$ is injective (triangular map), hence the center is global.
\end{theorem}

\begin{proof}
This is the construction in Example~8 of \cite{BraLliMer}. Since $f_{1x}=1$, $f_{1y}=\partial_y f_1$, $f_{2x}=\lambda$ and $f_{2y}=1+\lambda \partial_y f_1$, one computes
\begin{equation*}
\det Df=\begin{vmatrix} 1 & f_{1y}\\ \lambda & 1+\lambda f_{1y}\end{vmatrix}=1.
\end{equation*}
Thus \eqref{eq:iso-criterion} holds and the origin is an isochronous center of period $2\pi$. Degree counting shows $\deg H=2k$, hence $\deg(\dot x,\dot y)=2k-1$. Injectivity follows from the triangular structure.
\end{proof}

For small odd degrees this reproduces the classified normal forms: $k=2$
(cubic; Proposition~3 in \cite{BraLliMer}) and $k=3$ (quintic; Theorem~4 in
\cite{BraLliMer}).

\section{Quadratic--shear branch when $n\equiv3\pmod 4$}

The degree--7 classification in \cite{BraLliMer} unveils a second family built
from the quadratic shear
\begin{equation}
\label{eq:Q-def}Q:=y+\Gamma x^{2},\qquad\Gamma\neq0.
\end{equation}
In degree $7$ one has
\begin{equation}
\label{eq:septic-branch}f_{1}=x+\beta_{1} Q+\beta_{2} Q^{2},\qquad
f_{2}=Q+\lambda f_{1},
\end{equation}
which yields $\det Df\equiv1$ and is not linearly equivalent to the triangular
branch (Theorem~5 of \cite{BraLliMer}).

The same Jacobian cancellation persists for higher powers of $Q$, producing a
branch exactly when $n\equiv3\pmod 4$.

\begin{proposition}
[The $Q$--branch for $n=4m-1$]\label{prop:Q-branch} Let $m\geq2$ and define
\begin{equation}
Q:=y+\Gamma x^{2},\qquad f_{1}=x+\sum_{i=1}^{m}\beta_{i}\,Q^{\,i},\qquad
f_{2}=Q+\lambda f_{1}.\label{eq:Q-branch}%
\end{equation}
Then $\det Df\equiv1$. Consequently, $\mathcal{H}=\tfrac{1}{2}(f_{1}^{2}%
+f_{2}^{2})$ is a trivial isochronous Hamiltonian with $\deg f=2m$ and system
degree
\[
n=2\deg f-1=4m-1\equiv3\pmod 4.
\]

\end{proposition}

\begin{proof}
Let $S:=\sum_{i=1}^m i\beta_i Q^{i-1}$. Since $Q_x=2\Gamma x$ and $Q_y=1$, we have
\begin{equation*}
f_{1x}=1+Q_x S,\qquad f_{1y}=S,\qquad
f_{2x}=Q_x+\lambda f_{1x},\qquad f_{2y}=1+\lambda f_{1y}.
\end{equation*}
Therefore
\begin{align*}
\det Df
&= (1+Q_xS)\,(1+\lambda S)-S\,(Q_x+\lambda(1+Q_xS))\\
&= 1 + \lambda S + Q_x S + \lambda Q_x S^2 - S Q_x - \lambda S - \lambda Q_x S^2\\
&= 1.
\end{align*}
This is exactly the cancellation mechanism in the septic case \eqref{eq:septic-branch}. The degree claim follows.
\end{proof}

\begin{corollary}
[Alternation of branches]\label{cor:alternation} The $Q$--branch exists
precisely for $n\equiv3\pmod
4$ (i.e., $n=7,11,15,\dots$). For $n\equiv1\pmod 4$ (e.g., $n=5,9,13,\dots$)
it cannot occur, since $n=4m-1$ has no integer solution $m$. In particular: at
degree $9$ only the triangular branch exists; at degree $11$ both triangular
and $Q$ branches exist (with $m=3$). In degree $7$ the two branches are not
linearly equivalent \cite[Thm.~5]{BraLliMer}.
\end{corollary}

\section{Outline of the septic proof and its extension}

For degree $7$, \cite{BraLliMer} expands
\[
f_{1}=x+p_{2}+p_{3}+p_{4},\qquad f_{2}=y+q_{2}+q_{3}+q_{4},
\]
imposes $\det Df\equiv1$, and splits into a homogeneous chain of equations
(their Eqs.~(10)--(15)). By Lemma~1 (homogeneous degeneracy), either
$q_{i}=\lambda p_{i}$ (closing the triangular branch via Lemma~2), or a
special linear factor appears which, after a judicious linear change, produces
the quadratic shear $Q$ and the second branch \eqref{eq:septic-branch}. The
two branches are then shown to be inequivalent under linear transformations.
\medskip

For higher degrees the same chain persists with two more homogeneous layers
per extra degree of $f$, and the same dichotomy remains: the triangular branch
always closes; the $Q$--branch appears exactly when $\deg f$ is even (so that
$n=2\deg f-1\equiv3\pmod 4$), in agreement with
Corollary~\ref{cor:alternation}.

\section{Even degrees and nontrivial examples}

Even degrees admit no trivial centers under the hypotheses above (see the
discussion and citations in \cite{BraLliMer}). On the other hand,
\cite{BraLliMer} constructs \emph{nontrivial} isochronous centers in degrees
$6k+1 $, further underscoring that the trivial vs.\ nontrivial dichotomy is
essential in odd degrees.

\section*{Acknowledgements}

We thank Braun, Llibre and Mereu for their work, which inspired the present
note and from which we borrow both statements and methods.

\end{document}